\documentclass[12pt,reqno]{amsart} 

\usepackage{amssymb,bbm,enumitem,marginnote,url}
\usepackage{aliascnt}
\usepackage{doi}
\usepackage{pgfplots} 
\usepackage{booktabs}

\usepackage[margin=1.25in]{geometry}

\usepackage{psfrag}

\newcommand{\arxiv}[1]{%
 \href{https://arxiv.org/pdf/#1}{ArXiv:#1}}

\usepackage{color}

\newcommand\blue[1]{\textcolor{blue}{#1}}

\newtheorem{theorem}{Theorem}

\newaliascnt{lemma}{theorem}
\newtheorem{lemma}[lemma]{Lemma}
\aliascntresetthe{lemma}

\newaliascnt{proposition}{theorem}
\newtheorem{proposition}[proposition]{Proposition}
\aliascntresetthe{proposition}

\newaliascnt{corollary}{theorem}

\aliascntresetthe{corollary}

\newaliascnt{conjecture}{theorem}
\newtheorem{conjecture}[conjecture]{Conjecture}
\aliascntresetthe{conjecture}

\newaliascnt{openQ}{theorem}

\aliascntresetthe{openQ}

\newaliascnt{quest}{theorem}

\aliascntresetthe{quest}

\newaliascnt{questx}{conjx}

\aliascntresetthe{questx}

\theoremstyle{definition}

\newaliascnt{defn}{theorem}
\newtheorem{defn}[defn]{Definition}
\aliascntresetthe{defn}

\newaliascnt{example}{theorem}

\aliascntresetthe{example}

\newaliascnt{rem}{theorem}

\aliascntresetthe{rem}

\makeatletter
\def\tagform@#1{\maketag@@@{\ignorespaces#1\unskip\@@italiccorr}}
\let\orgtheequation\theequation
\def\theequation{(\orgtheequation)}
\makeatother
\def\equationautorefname~{}

\newcommand{\B}{{\mathbb B}}

\newcommand{\R}{{\mathbb R}}

\newcommand{\RP}{{\mathbb{RP}}}

\newcommand{\bu}{\mathbf u}
\newcommand{\bv}{\mathbf v}

\newcommand{\V}{{\text{Vol}}}

\let\oldmarginnote\marginnote
\renewcommand{\marginnote}[1]{\oldmarginnote{\tiny \blue{#1}}}

\begin{document}
\title[Second Laplacian eigenvalue on real projective space]{Upper Bound on the Second Laplacian Eigenvalue on Real Projective Space}

\keywords{Spectral theory, shape optimization}
\subjclass[2020]{\text{Primary 35P15. Secondary 58C40, 58J50}}

	\begin{abstract}
In this paper, we prove an upper bound on the second non-zero Laplacian eigenvalue on $n$-dimensional real projective space. The sharp result for 2-dimensions was shown by Nadirashvili and Penskoi and later by Karpukhin when the metric degenerates to that of the disjoint union of a round projective space and a sphere. That conjecture is open in higher dimensions, but this paper proves it up to a constant factor that tends to 1 as the dimension tends to infinity. Also, we introduce a topological argument that deals with the orthogonality conditions in a single step proof. 
	\end{abstract}
	
\author[]{Hanna N. Kim}
\address{Department of Mathematics, University of Illinois, Urbana--Champaign, IL 61801, U.S.A.}
\email{nekim2@illinois.edu}

	\maketitle 
	
	
\section{\bf Introduction and results} \label{sec:intro}

For $n$-dimensional real projective space $\RP^n$, consider a round metric $g$ and its conformally equivalent metrics $wg$ where $w$ is the conformal factor. Throughout the paper, the metric is normalized so that the volume $\V(\RP^n, wg) \equiv \V_n(w)$ is equal to the volume $\V(\RP^n, g)$ of the round projective space. The Laplace–Beltrami operator $-\Delta_{wg}$ on $(\RP^n, wg)$ has a discrete sequence of eigenvalues:
$$
0 = \lambda_0(w) < \lambda_1(w) \leq \lambda_2(w)  \leq \dots \to \infty.
$$

This paper provides an upper bound for the second nonzero eigenvalue $\lambda_2(w)$ for all $\RP^n$, $n \geq 2$. The strategy of the proof is to find $n+1$ trial functions, which we construct using Veronese embeddings to map projective spaces to higher dimensional spheres followed by a fold and a M\"{o}bius transformation. We rely on topological arguments to show that the trial functions are orthogonal to the first two modes.

\subsection{\bf Results on the second eigenvalue}  
\label{sec:results}	

In this section, we review isoperimetric inequalities of the Laplacian eigenvalues related to spheres and projective spaces, and state our main result.

\begin{theorem}[Li and Yau \cite{LY82} for $n=2$ and El Soufi and Ilias \cite{EI86} for $n \geq 2$; first eigenvalue on $\RP^n$] 
Assume that the metric $wg$ is normalized so that $\V_n(w)=\V_n(1)$. Then the first nonzero eigenvalue of $-\Delta_{wg}$ satisfies,
$$
\lambda_1(w) \leq \lambda_1(1) = 2n+2.
$$
\end{theorem}
The upper bound is attained when the metric is round on $\RP^n$. For the next higher eigenvalue on 2-dimensional projective space, we have the following result.

\begin{theorem}[Nadirashvili and Penskoi \cite{NA18}; second eigenvalue on $\RP^2$]\label{thmsharp}
Assume that the metric $wg$ is normalized so that $\V_2(w)=\V_2(1)$ on $\RP^2$. Then the second nonzero eigenvalue of $-\Delta_{wg}$ satisfies,
$$
\lambda_2(w) \leq 10.
$$
\end{theorem}
The upper bound can be obtained by a sequence of metrics approaching to that of the disjoint union of a round projective space and a sphere having ratio of radii $\sqrt{6} : \sqrt{2}$. A natural conjecture is that the upper bound can be extended to all dimensions. 
\begin{conjecture}[Second eigenvalue on $\RP^n$, $n \geq 3$] \label{conj}
Assume that the metric $wg$ is normalized so that $\V_n(w)=\V_n(1)$ on $\RP^n$. Then the second nonzero eigenvalue of $-\Delta_{wg}$ satisfies,
\begin{equation}
\lambda_2(w) < ((2n+2)^{n/2}+2n^{n/2})^{2/n}.
\label{mainequ}
\end{equation}
\end{conjecture}
The upper bound can be obtained by a sequence of metrics approaching to that of disjoint union of a round projective space and a round sphere having ratio of radii $\sqrt{2n+2} : \sqrt{n}$.

We prove the following bound on $\lambda_2(w)$, with a right hand side that is larger than in the conjecture.
\begin{theorem}[Second eigenvalue on $\RP^n$, $n \geq 3$] \label{thm}
Assume that the metric $wg$ is normalized so that $\V_n(w)=\V_n(1)$ on $\RP^n$. Then the second nonzero eigenvalue of $-\Delta_{wg}$ satisfies,
\begin{equation}
\lambda_2(w) <2^{2/n}(2n+2).
\label{mainequ2}
\end{equation}
\end{theorem}
In dimension 2, the above theorem gives $\lambda_2(w) <12$, which is weaker than the sharp bound 10 in \autoref{thmsharp}. In \autoref{ratio}, we confirm that \autoref{conj} is stronger than \autoref{thm}, and that the theorem is asymptotically sharp as $n$ tends to $\infty$ because the ratio between the two upper bounds approaches 1. 

To conclude this introduction, we summarize some related literature.

\subsection*{Higher eigenvalues on the 2-dimensional real projective space $\RP^2$} 
For the higher eigenvalues on $\RP^2$, it is known by Karpukhin \cite{K19} that $k$-th eigenvalue has a sharp upper bound when a sequence of metrics converges to the disjoint union of a real projective space and $k-1$ identical round spheres, where the ratio of radii between the projective space and the spheres is $\sqrt{6} : \sqrt{2}$.

\subsection*{First and second eigenvalue on $\mathbb{S}^n$}
For the first eigenvalue $\lambda_{1}$, Hersch  in 2-dimensions \cite{H70} and El Soufi and Ilias \cite{EI86} in higher dimensions proved that the sharp upper bound of the first nonzero eigenvalue $\lambda_1$ is attained by the round metric (within that conformal class). They found trial functions orthogonal to the constant by composing the eigenfunctions of the round sphere (which are the $m+1$ coordinate functions) with a M\"{o}bius transformation to move the center of the mass to the origin. 

For maximizing $\lambda_2$ on the 2-sphere, Nadirashvili \cite{N02} and later Petrides \cite{P14} showed that there exists a maximizing sequence of metrics degenerating to a disjoint union of two equal round spheres. In my work \cite{K20}, the analogous result for the second eigenvalue $\lambda_{2}$ on the higher-dimensional sphere $\mathbb{S}^n$ is proved. Inspired by the prior works on constructing trial functions, the method in this paper relies on building trial functions satisfying the orthogonality conditions by composing the eigenfunctions of the round sphere with a fold map and a M\"{o}bius transformation.

\subsection*{Higher eigenvalues on 2-sphere}
Nadirashvili and Sire \cite{NS17} confirmed the analogous result for the third nonzero eigenvalue $\lambda_3$. For all eigenvalues, by using the work of Petrides \cite{P17}, Karpukhin, Nadirashvili, Penskoi and Polterovich \cite[Theorem 1.2]{KNPP19} showed that the $k$-th eigenvalue on ${\mathbb S}^{2}$ is maximal when a sequence of metrics degenerates to the disjoint union of $k$ identical round spheres.
%
%

\section{\bf Overview of the proof}  \label{sec:overview}	

Let us preview the trial functions, which are the components in $\R^{m(n)}$ of the vector valued map:
\begin{align}
 (Y \circ T_{-c} \circ F_{H} \circ \Phi)(y) \label{vec}, \qquad y \in \RP^n,
\end{align}
where $T_{-c}$ is a M\"{o}bius transformation on the ball $\B^{m(n)}$, $ F_{H}$ is a fold map on ${\mathbb S}^{m(n)-1}$, $\Phi \equiv \Phi_n:  \RP^n \to {\mathbb S}^{m(n)-1} $ is a generalized Veronese map and $Y(y)=y$ is the identity map on the sphere. Note that each component $Y_{j}(y)=y_{j}$ is an eigenfunction for the round sphere. We drop the identity map ``$Y$" in the work that follows since its role in \autoref{vec} is mainly to emphasize that the trial functions are the $m+1$ components of the vector field in \autoref{vec}.

In this paper, we generalize Veronese embedding to all dimensions by induction, building on the work of Zhang \cite{Z09}. We adapt the conformal factors so that the images of Veronese embeddings lie on higher dimensional spheres of radius 1. The construction of trial functions relies on my earlier work \cite{K20}. We build $m(n)$-number of trial functions that satisfy the conditions for the variational characterization of the second nonzero eigenvalue: orthogonality to the constant and to the first excited state on $\RP^n$ with the metric $wg$.


%
%


Let us mention again some relavant previous works. For $\RP^2$,  Li and Yau \cite{LY82} used the well-known Veronese embedding, which is a minimal immersion of $\RP^2$ to ${\mathbb S}^4$, to show that the round metric maximizes the first eigenvalue. El Soufi and Ilias \cite{EI86} generalized the result to $\RP^n$ for all $n \geq 2$, showing that since $\RP^n$ can be minimally immersed to a higher dimensional sphere by its first eigenfunctions, the standard metric induces the sharp upper bound.  We combine this idea of mapping real projective space to a higher dimensional sphere with the method of constructing trial functions from my previous paper \cite{K20}.

\section{\bf Veronese embedding of projective space into a sphere } \label{sec:veronese}

Our goal in this section is to construct a “generalized Veronese” map $\Phi_{n}: \RP^{n} \to {\mathbb S}^{m(n)-1}$ that is a conformal embedding. We first define $\Phi_{n}: \R^{n+1} \to \R^{m(n)}$, and show that when the map is restricted on the real projective space, the image lies in a higher dimensional sphere. 

We denote $m(n)$ as the multiplicity of the first eigenvalue of the round $\RP^n$,  
$$m(n) = \frac{n(n+3)}{2}$$
for all $n \geq 1$ (see \cite[Corollary 7.4.3]{T10}). And note that the first eigenvalue of the round $\RP^n$ is $ 2n+2$. Readers may wish to skip this section since the explicit formulas are not needed and we only need the conformal embedding property in \autoref{phiconformal}.

\begin{defn}
For $n=1$, define $\Phi_{1}: \R^{2} \to \R^{2}$ by
$$
\Phi_{1}(x_1,x_2) = \left(2x_1 x_2, x_1^2 -x_2^2 \right).
$$
For $n=2$, define $\Phi_{2}: \R^{3} \to \R^{5}$ by
$$
\Phi_{2}(x_1,x_2,x_3) = \sqrt{3} \left(x_1 x_2, \dfrac{1}{2}(x_1^2 -x_2^2) , x_1 x_3, x_2 x_3,  \dfrac{1}{2 \sqrt{3}}(x_1^2 +x_2^2 - 2 x_3^2) \right).
$$
For $n \geq 2$, define $\Phi_{n}: \R^{n+1} \to \R^{m(n)}$ inductively by
\begin{equation}
\label{veronese}
\begin{split}
\Phi_{n}(x_1, \cdots ,x_{n+1}) = &a_n \left( \frac{1}{a_{n-1}}\Phi_{n-1}(x_1, \dots ,x_{n}), x_1x_{n+1}, \dots \right. \dots, x_{n} x_{n+1}, 
\\&  \left. \frac{1}{na_n}\left(x_1^2 + x_2^2 +\dots +x_{n}^2 -nx_{n+1}^2 \right) \right),
\end{split}
\end{equation} 
where the constant is 
$$a_n = \left( \frac{2n+2}{n} \right)^{\! \! 1/2}.$$ 
\end{defn}

As a remark, when $n=1$, one can see that the image consists of coordinates of the square of the complex number $z = x_1 +i x_2$. Note that when $n=2$, the map $\Phi_2$ is the well-known Veronese conformal embedding. The inductive definition of $\Phi_n$ for $n \geq 2$ involves a normalized  $\Phi_{n-1}$, the cross terms $x_ix_{n+1}$ and then the final term. Notice that $\Phi_n(-x) = \Phi_n(x)$ and so $\Phi_n$ is well-defined on $\RP^n$.

\begin{proposition}
\label{phiconformal}
If the domain of $\Phi_{n}$ in \autoref{veronese} is restricted from $\R^{n+1}$ to the projective space $\RP^n$, then
$$ \Phi_{n} : \RP^{n} \to {\mathbb S}^{m(n)-1}$$
 is a conformal embedding with the conformal factor as $a_n$.
\end{proposition}

The proposition follows by combining the next two lemmas.

\begin{lemma}
The Veronese map $\Phi_{n}: \R^{n+1} \to \R^{m(n)}$ satisfies $|\Phi_{n}(\mathbf x) | =|\mathbf x|^2$. In particular, the image of the map $\Phi_{n}$ restricted to the projective space $\RP^{n}$ lies in the unit $(m(n)-1)$-sphere and the map $\Phi_{n}$ is injective on $\RP^{n}$.
\end{lemma}

\begin{proof}
One can check that the result holds easily for $n=1$. Now, assume that the lemma holds for $n-1 \geq 1$. We show that it also holds for $n$.

We have $|\Phi_{n-1}(x_1, \dots, x_{n})| = |(x_1, \dots, x_{n})|^2$ by the induction hypothesis. The square of $|\Phi_{n}|$ is written as follows:
\begin{align*}
|\Phi_{n}(x_1, \dots, x_{n+1})|^2 &= \dfrac{2n+2}{n} \left( \dfrac{n-1}{\lambda_{n-1}}(x_1^2+ \dots +x_n^2)^2+ (x_1^2 + \dots + x_{n}^2)x_{n+1}^2 \right.
\\& \left.+\dfrac{1}{n(2n+2)}(x_1^2 +  \dots  \dots +x_{n}^2 -n x_{n+1}^2)^2 \right).
\end{align*}

After replacing $x_1^2 + \dots +x_n^2$ with $|\mathbf{x}|^2 -x_{n+1}^2$ and simplifying the above, one finds $|\Phi_n(x_1, \dots, x_{n+1})|^2 = |\mathbf{x}|^4$, where $\mathbf{x} = (x_1, \dots, x_{n+1})^T$ is a column vector.





One can easily see that the map is injective on $\RP^1$ since on the complex plane, the map can be written as $\Phi_1(z)=i \bar{z}^2$ which is injective on $\RP^1$. Using a short induction argument, we find $\Phi_{n}$ is injective on $\RP^{n}$.

\end{proof}
Next, we compute the properties of $D\Phi_n$ and show the Veronese map $\Phi_n$ is conformal on $\RP^n$.
\begin{lemma}
For $n \geq 1$, 
\begin{align*}
(D\Phi_{n}(\mathbf{x}))^{T}D\Phi_{n}(\mathbf{x}) = \frac{2(n+1)}{n}| \mathbf x |^2 I +\frac{2(n-1)}{n} \mathbf{x x}^T,  \qquad \mathbf{x} \in \R^{n+1}.
\end{align*}
In particular, $\Phi_{n}$ restricted to the projective space $\RP^{n}$ is a conformal map into an $(m(n)-1)$-sphere with conformal factor $a_n$.
\end{lemma}

\begin{proof}
We prove by induction. For $n=1$, 
\begin{gather*}
D\Phi_{1}^{T}D\Phi_{1}
=4
\begin{pmatrix}
x_2 & x_1 \\
x_1 & -x_2
\end{pmatrix}
\begin{pmatrix}
x_2 & x_1 \\
x_1 & -x_2
\end{pmatrix} 
=4
\begin{pmatrix}
x_1^2 + x_2^2 & 0 \\
0 & x_1^2+ x_2^2
\end{pmatrix} 
= 4|\mathbf{x}|^2I,
\end{gather*}
where the conformal factor is $a_1=2$. Assume that the lemma holds for $n-1 \geq 1$. Recall that $a_n =\sqrt{2(n+1) /n}  $ and define $b_n = \sqrt{2(n-1)/n} $. By the induction hypothesis,
\begin{align*}
(D\Phi_{n-1}(\mathbf{x}_n))^TD\Phi_{n-1}(\mathbf{x}_n) = a_{n-1}^2 |\mathbf{x}_{n}|^2 I_{n} +b_{n-1}^2\mathbf{x}_{n}\mathbf{x}_{n}^T, 
\end{align*}
where $I_{n}$ is an $n \times n$ identity matrix and $\mathbf{x}_n =(x_1, \dots, x_{n})^T \in \R^{n}$. Let $\mathbf{x}= (\mathbf{x}_n, x_{n+1})^T \in \R^{n+1}$. The derivative of $\Phi_n$ is written as a $3 \times 2$ block matrix,

\begin{gather*}
D\Phi_{n}(\mathbf{x}) = a_n
\begin{pmatrix}
a_{n-1}^{-1}D\Phi_{n-1} & \mathbf{0}  \\
x_{n+1}I_{n}& \mathbf{x}_{n} \\
d_{n}\mathbf{x}_{n}^T & -b_{n+1} x_{n+1}
\end{pmatrix},
\end{gather*}
where $d_{n} = \sqrt{2/n(n+1)}$. We write $D\Phi_{n}^{T}D\Phi_{n}$ as a $2 \times 2$ block matrix,
\begin{gather}
D\Phi_{n}^{T}D\Phi_{n}=
a_n^2
\begin{pmatrix}
a_{n-1}^{-2}D\Phi_{n-1}^TD\Phi_{n-1} + x_{n+1}^2 I_n + d_n^2 \mathbf{x}_{n}\mathbf{x}_{n}^T & (1-b_{n+1}d_n) x_{n+1}\mathbf{x}_{n} \\[12pt]
(1-b_{n+1}d_n) x_{n+1}\mathbf{x}_{n}^T & \mathbf{x}_{n}^T \mathbf{x}_{n} +b_{n+1}^2 x_{n+1}^2
\end{pmatrix}.\label{block}
\end{gather}
Next, we use the induction hypothesis to expand $D\Phi_{n-1}^T D\Phi_{n-1}$ on the upper left term in \autoref{block}, having
\begin{align*}
a_{n-1}^{-2}D\Phi_{n-1}^TD\Phi_{n-1} + x_{n+1}^2 I_n + d_n^2 \mathbf{x}_{n}\mathbf{x}_{n}^T &=|\mathbf{x}_{n}|^2 I_{n} +  a_{n-1}^{-2}b_{n-1}^2\mathbf{x}_{n}\mathbf{x}_{n}^T + x_{n+1}^2 I_n + d_n^2 \mathbf{x}_{n}\mathbf{x}_{n}^T\\
&=|\mathbf{x}|^2  I_{n} +(a_{n-1}^{-2} b_{n-1}^2+ d_n^2 )\mathbf{x}_{n}\mathbf{x}_{n}^T.
\end{align*}
Hence we can write the matrix in \autoref{block} as follows, 
\begin{gather}
D\Phi_{n}^{T}D\Phi_{n} =
a_n^2 |\mathbf{x}|^2 I_{n+1} +
a_n^2
\begin{pmatrix}
 ( a_{n-1}^{-2} b_{n-1}^2+ d_n^2 )\mathbf{x}_{n}\mathbf{x}_{n}^T & (1-b_{n+1}d_n) x_{n+1}\mathbf{x}_{n} \\[12pt]
(1-b_{n+1}d_n) x_{n+1}\mathbf{x}_{n}^T & (b_{n+1}^2-1) x_{n+1}^2
\end{pmatrix} \label{matrix} .
\end{gather}

We only need to check all four coefficients of the second matrix in \autoref{matrix} are equal to $b_n^2/ a_n^2$, which is straightforward.  So, the conclusion of the Lemma holds for all $n$.

Finally, we show that the map is conformal on $\RP^n$. Let $\mathbf{x} \in \mathbb{S}^n$. For all $\bu, \bv \in \R^{n+1}$ that are tangential to the sphere at $\mathbf{x}$, 
\begin{align*}
\bu^{T}D\Phi_{n}^{T}D\Phi_{n} \bv &= \bu^{T}\left(\frac{2(n+1)}{n}| \mathbf x |^2 I +\frac{2(n-1)}{n} \mathbf{x x}^T \right)\bv 
\\&= a_n^2\bu^{T}\bv,
\end{align*}
where $|\mathbf{x}|^2=1$ and the term including $\bu^T\mathbf{x}$ and $\mathbf{x}^T\bv$ vanished since the vectors $\bu , \bv$ of the tangent space are orthogonal to the point $x$ on the sphere. This proves that the generalized Veronese map $\Phi_n$ on $\RP^n$ preserves angles and scales distances by the conformal factor $a_n$. 

\end{proof}

\section{\bf Trial Functions and Orthogonality} \label{sec:trial}

Relying on the techniques from Freitas and Laugesen \cite{FL20} for domains in hyperbolic space, I constructed the trial functions on the sphere in my previous paper \cite{K20}. In this paper, we adapt the method by first composing the Veronese embedding, since the embedding takes $\RP^n$ into ${\mathbb S}^{m(n)-1}$. Then we can apply a similar argument from the sphere: fold the image across some spherical cap and compose it with a suitable M\"{o}bius transformation. Now, we reintroduce the definitions of these maps to help the reader's understanding.

\subsection{M\"{o}bius transformations} 
Write $\B^{m}$ for the unit ball in $\R^{m}$ and ${\mathbb S}^{m-1} = \partial \B^{m}$. Define the M\"{o}bius transformations on the closed ball \cite[eq.(2.1.6)]{S16}, parametrized by $x \in \B^{m}$, as 
\[
T_{x}: \overline{\B}^{m} \to \overline{\B}^{m},
\]
\begin{equation}
\label{mobius}
T_{x}(y)=\frac{(1+2x \cdot y +|y|^{2})x + (1-|x|^{2})y}{1+2x \cdot y +|x|^{2}|y|^{2}}
, \qquad  y \in \overline{\B}^{\, m}.
\end{equation}
As a remark, $T_{0}$ becomes just the identity map on the ball. Also, $T_{x}(0)=x$, $T_{-x}=(T_x)^{-1}$, and $T_{x}$ maps ${\mathbb S}^{m-1}$ to ${\mathbb S}^{m-1}$, fixing the points $y= \pm x/|x|$. 

\subsection{Spherical caps, reflection, and folding} \label{maps}
The following material is drawn from \cite{K20} and is included here for the reader's convenience. We describe spherical caps. For any unit vector $p$ on the sphere, the closed hemisphere can be written as
\[
H_{p}=\{y \in {\mathbb S}^{m-1} : y \cdot p \leq 0\}, \qquad p \in {\mathbb S}^{m-1}.
\]
Next, define the spherical caps as the image of the hemisphere $H_p$ under some M\"{o}bius transformation:
\[
H \equiv H_{p,t} = T_{pt}(H_{p}), \qquad p \in {\mathbb S}^{m-1}, \qquad t \in [0,1).
\]
Here, the M\"{o}bius transformation sends boundaries to boundaries, that is, $T_{pt} (\partial H_{p})=\partial H_{p,t}$. Then we can write the spherical cap explicitly as
\[
H_{p,t}= \left\{ y \in {\mathbb S}^{n} : y \cdot p \leq \frac{2t}{1+t^{2}} \right\},
\]
which can be easily calculated by \autoref{mobius}. Note that as $t$ approaches to $1$, the spherical cap $H_{p,t}$ covers almost all the sphere toward $p$.

Given $b\in \R^{m} \setminus\{0\}$, the reflection $R_b$ in the hyperplane through the origin and perpendicular to the vector $b$ is defined as,
\[
R_{b}(y)=y -2 \frac{(y \cdot b)}{|b|^2 }b, \qquad y \in {\mathbb S}^{m-1}.
\]
By conjugation, we can define a reflection map across the boundary of the general spherical cap $H_{p,t}$: let
\begin{align*}
\label{conjug}
R_{H} \equiv R_{p,t}= T_{pt} \circ R_{p} \circ ( T_{pt})^{-1}: {\mathbb S}^{m-1} \to {\mathbb S}^{m-1}.
\end{align*} 
Note that $R_{H_{p,t}}(p)=-p$, which says that the reflection map sends $p$ to its antipodal point. Lastly, we define a ``fold map" that reflects the complement of the spherical cap across the boundary:
\[
F_{H}(y) \equiv F_{p,t}(y)=
\begin{cases}
y , & y \in H , \\
	R_{H}(y) , & y \in {\mathbb S}^{m-1} \setminus H.
\end{cases}
\]
Observe that the image of the boundary of the spherical cap $\partial H$ is itself. For simplicity, we use different notations $F_{H}$ and $F_{p,t}$ to denote the same map when it is clear from the context (similarly for $R_{H}$). 

\subsection{Center of mass and trial functions}
A point $c \in \B^{m}$ is the (hyperbolic) center of mass of a Borel measure $\mu$ on the sphere if
\begin{align}
\int_{{\mathbb S}^{m-1}} T_{-c}(y)  \, d\mu(y) =0. \label{center}
\end{align}

The trial functions are defined to be the components in $\R^{m(n)}$ of the vector valued map:
\begin{align}
y \mapsto  (T_{-c} \circ F_{H} \circ \Phi_n)(y), \qquad y \in \RP^n, \label{trial}
\end{align}
where $T_{-c}$ is a M\"{o}bius transformation on the ball $\B^{m(n)}$, $ F_{H}$ is a fold map on ${\mathbb S}^{m(n)-1}$, and $\Phi \equiv \Phi_n:  \RP^n \to {\mathbb S}^{m(n)-1} $ is the generalized Veronese map defined in \autoref{sec:veronese}. How the center of mass $c$ and spherical cap $H$ are chosen is explained below.


Next, we present two different topological proofs that the trial functions satisfy the orthogonality conditions. The first argument is based on the idea of Hersch \cite{H70} and the reflection symmetry lemma by Petrides \cite{P14}. The second argument is essentially from a result of Karpukhin and Stern \cite[Lemma 4.2]{KS20}. We give a new proof.
\subsection{Orthogonality of the trial functions by two-step proof} \label{sec: twostep}

Based on the construction above, we need to show that the trial functions are orthogonal to the constant and the first eigenfunction with respect to the metric $wg$, so that we can use the variational characterization of the second eigenvalue. The first proof of orthogonality of trial functions \autoref{trial} proceeds similarly to to my previous work \cite[p.\,3506]{K20}, as we now explain. 

 Define a push-forward measure $\mu$ on $\mathbb{S}^{m(n)-1}$ by $ \mu=  (F_H \circ \Phi)_{*} v_{wg}$ where $v_{wg}$ is the volume measure on $\RP^n$ with respect to the metric $wg$. 
A center of mass $c=c_H$ in \autoref{center} exists by Hersch's lemma; see Laugesen \cite[Corollary 5]{L20b} for the precise statement. The assumptions of the corollary are satisfied by extending $\mu$ outside $ F_{H}(\Phi(\RP^n))$ to all of the sphere with zero, noting that this extended pushforward measure is a finite Borel measure and $ 0 = \mu(\{y\}) < \frac{1}{2} \mu({\mathbb S}^{m(n)-1}) $ for all $y \in {\mathbb S}^{m(n)-1}$. The result of the corollary gives the existence and uniqueness of the center of mass $c_H = c_{p,t}$, and the center of mass depends continuously on the parameters of the spherical cap, $(p,t) \in {\mathbb S}^{m(n)-1} \times [0,1)$. 

We later need the continuity of the center of mass as $t \to 1$. The M\"{o}bius transformation and the fold map do not extend continuously when $t=1$, but nonetheless, the center of mass $c_{p,t}$ converges to a point $c(w)$ which depends only on the measure and is independent of $p$. See \cite[p.\,3506]{K20}. The underlying point is that the push-forward measure is weakly convergent as $t \to 1$.

The orthogonality of trial functions to the first excited state $f$ of $-\Delta_{wg}$ on $\RP^n$ requires that the following vector field vanishes at some point $(p,t) \in {\mathbb S}^{m(n)-1} \times [0,1)$:
\begin{equation}
\label{first}
V(p,t) = \int_{ \RP^n}^{} T_{-c_{H}}(F_H(\Phi(y))) f(y)\, dv_{wg}(y),
\end{equation}
where $H=H_{p,t}$. Note that from the continuous dependence on the parameters, the vector field \autoref{first} is continuous. Since the image of the Veronese map is in the sphere, the argument in my previous paper \cite[Section 2.8]{K20}, which is based on the topological argument by Petrides \cite[claim 3]{P14}, applies with obvious changes. We skip the details and state the result below from \cite[Proposition 8]{K20}.

\begin{proposition}[Vanishing of the vector field]\label{vanish}
$V(p,t)=0$ for some $p \in {\mathbb S}^{m(n)-1}$ and $t \in [0,1]$.
\end{proposition}
This result finishes the first argument for the orthogonality conditions. Before we proceed to the second argument, I would like to introduce different approaches to the orthogonality conditions for trial functions. Petrides \cite{P14} showed that a map with reflection symmetry has nonzero degree, leading to the two-step argument above. Freitas and Laugesen \cite{FL20} gave a new proof of Petrides's Lemma by a global approach using the de Rham definition of the degree. 

Karpukhin and Stern \cite[Lemma 4.2]{KS20} relied instead on the Lefschetz--Hopf fixed point theorem, obtaining both orthogonality conditions in a one-step proof. In the next section, we give a proof similar to their lemma by adapting a method for proving the Borsuk--Ulam theorem. A third approach by Bucur, Martinet and Nahon \cite{BMN22} relies on the index theorem to give a one-step proof in a closely related orthogonality situation.

\subsection{Orthogonality of the trial functions by one-step proof} \label{onestep}

We rely on \autoref{topdeg} in \autoref{sec:top2} to prove that for some choice of parameters, the trial functions are orthogonal to both the constant and the first excited state. Compared to the previous vector field defined in \autoref{first}, we assume now that the center of mass is an independent variable in the following vector field, 
\[
V(x,p,t) = \left( \int_{ \RP^n}^{} T_{-x}(F_H(\Phi( y))) \, dv_{wg}, \int_{ \RP^n}^{} T_{-x}(F_H(\Phi( y))) f(y)\, dv_{wg} \right),
\]
which is a map from $\B^{m(n)} \times {\mathbb S}^{m(n)-1} \times[0,1)$ to $\R^{2m(n)}$. Here $H = H_{p,t}$. Now, we extend $V$ continuously to $t=1$ and $|x|=1$. When $t =1$, the vector field is continuous and independent of $p$, by a similar argument as in the previous section. As $x \to \tilde{x} \in \mathbb{S}^{m(n)-1}$, the vector field can be extended continuously by the dominated convergence similarly and then the vector field becomes independent of $p$ and $t$. In fact, the vector field at $x= \tilde{x}$ is 
$$V(\tilde{x}, p , t) = \V_n(w)(-\tilde{x}, 0)$$ 
since the M\"{o}bius transformation degenerates and $\int_{\RP^n} f(y) \, dv_{wg} =0$. In the next proposition, we again use the fact that the degree is a homotopy invariant to show that the vector field vanishes at some choice of parameters. 

\begin{proposition}[Vanishing of the vector field]\label{vanish2}
$V(x,p,t)=0$ for some $(x,p) \in \overline{\B}^{m(n)} \times {\mathbb S}^{m(n)-1}$ and $t \in [0,1]$. 
\end{proposition}

\begin{proof}
Let $m =m(n)$. We endow the following equivalence relation on $\overline{\B}^{m} \times {\mathbb S}^{m-1}$: For any $x \in \partial \overline{\B}^{m}$ and $p,q \in {\mathbb S}^{m-1}$, we say $(x,p) \sim (x, q)$. Define a homeomorphism $\overline{\B}^{m} \times {\mathbb S}^{m-1}/{\mathord {\sim }}  \to  {\mathbb S}^{2m-1}$ as $(x,p) \mapsto (\sqrt{2-|x|^2}x, (|x|^2-1)p)=(a,b)$. (This homeomorphism makes the boundary points of the ball collapse onto the points in the sphere such that the second part of the coordinates are zero.)  The inverse parameters are
$$x(a) =- \frac{a}{\sqrt{1+|b|}} \qquad \text{and} \qquad  p(b) =-\frac{b}{|b|}.$$ 
When $b \neq 0$, $x(a)$ only depends on $a$ since $|b|^2 =1-|a|^2$. When $b=0$, we have $x(a)= -a$ and $p(b)$ is undefined. 

Next, we precompose V with the inverse of the homeomorphism and denote the map as 
$$\widetilde{V}(a,b, t)=V(x(a), p(b), t), \qquad (a,b) \in \mathbb S^{2m-1},$$ 
where for $b=0$, we have $|a|=1$, $x(a) =-a$, $|x(a)|=1$, and so $\widetilde{V}(a,0, t)=\V_n(w)(a,0)$. 

Suppose $V(x,p,t)$ does not vanish, so that $\widetilde{V}$ does not vanish. We obtain a contradiction later. For each $t$, let $W_t(a,b):=\widetilde{V}(a,b,t)/|\widetilde{V}(a,b, t)|$ be a map from ${\mathbb S}^{2m-1}$ to itself. When $b=0$, note $W_t(a,0) = (a,0)$.

When $t=1$, the fold map becomes identity on all of the sphere except at one point, so we have
\[
W_1(a,b) = \left( \int_{ \RP^n}^{} T_{-x(a)}(\Phi(y)) \, dv_{wg}, \int_{ \RP^n}^{} T_{-x(a)}(\Phi(y)) f(y)\, dv_{wg} \right).
\]
The map $W_1$ is not surjective onto $\mathbb{S}^{2m-1}$ since it is smooth and the right side depends only on the $m$ dimensional parameters $a$, while $2m-1>m$. Hence, $W_1$ is homotopic to a constant map and so has degree zero, which implies that $W_0$ also has degree zero.

When $t=0$, the map $(a,b) \mapsto W_0(a,b)$ satisfies the reflection symmetry condition \autoref{refsym} below, by the following calculation. When $b \neq 0$, we calculate that
\begin{align}
&\widetilde{V}(R_b(a),-b,0) \nonumber \\
&= \left( \int_{ \RP^n}^{} T_{-x(R_b(a))}(F_{p(-b),0}(\Phi( y))) \, dv_{wg}, \int_{ \RP^n}^{} T_{-x(R_b(a))}(F_{p(-b),0}(\Phi( y))) f(y)\, dv_{wg} \right)  \nonumber\\
&= \left( \int_{ \RP^n}^{} T_{-R_b(x(a))} (R_b F_{p(b),0}(\Phi( y))) \, dv_{wg}, \int_{ \RP^n}^{}T_{-R_b(x(a))} (R_b F_{p(b),0}(\Phi( y))) f(y)\, dv_{wg} \right)\nonumber \\
& \hspace{2cm} \text{since $x(R_b(a)) = R_b(x(a))$ by linearity of $R_b$, and $F_{p(-b),0} = R_b F_{p(b),0}$}\nonumber\\
&= \left( R_b\int_{ \RP^n}^{} T_{-x(a)}(F_{p(b),0}(\Phi( y))) \, dv_{wg}, R_b \int_{ \RP^n}^{} T_{-x(a)}(F_{p(b),0}(\Phi( y))) f(y)\, dv_{wg} \right)  \nonumber \\
& \hspace{2cm} \text{by the property $T_{R_b  a}(R_b y) = R_b T_{a}(y)$ in \cite[p. 3507]{K20}} \nonumber\\
&= (R_b \times R_b) \widetilde{V}(a,b,0).\nonumber
\end{align}
Now the reflection symmetry condition for $W_0$ follows by dividing each side by its norm. When $b=0$, $W_0(a,0) = (a,0)$. Thus \autoref{refsym} holds for $W_0$.

Hence, the map $W_0$ has nonzero degree by \autoref{topdeg}, which is a contradiction.
\end{proof}

\section{\bf Calculation of the degree of the self-maps } \label{sec:top2}

In this section, we calculate the degree of the continuous maps between odd-dimensional spheres with reflection symmetry. Recall the following facts about the toplogical degree of continuous maps on a sphere. See Outerelo and Ruiz \cite[Chapter IV.4]{OR09}. Let  $\varphi :\overline\B^{n+1} \to \R^{n+1}$ be a continuous map such that $\varphi({\mathbb S}^n) \subset \R^{n+1} \setminus \{0\}$. Define a continuous map $\phi: {\mathbb S}^n \to {\mathbb S}^n$ as $\phi(x) =\varphi(x)/ |\varphi(x)|$ for $x \in {\mathbb S}^n$. It is well-known that for any point $p \in {\mathbb S}^n$, $d(\varphi,\overline\B^{n+1}, 0)=\text{deg}(\phi, {\mathbb S}^n, p)$. Moreover, since the degree of $\phi$ is consistent on any $p \in {\mathbb S}^n$, we may write $\text{deg}(\phi)=\text{deg}(\phi, {\mathbb S}^n, p)$. The main goal of this section is to show the following theorem. 

\begin{theorem} \label{topdeg}
Let $\phi: {\mathbb S}^{2n+1} \to {\mathbb S}^{2n+1}$ be a continuous map and assume that the map satisfies the following reflection symmetry property:
\begin{equation}
\label{refsym}
\begin{split}
(R_b \times R_b) \phi(a,b) &= \phi(R_b(a), -b) \qquad\text{when $b \neq 0$,} \\
\phi(a,0) &= (a,0) \qquad\qquad\quad\, \text{when $b = 0$,}
\end{split}
\end{equation}
for all $a, b \in \R^{n+1}$ with $(a,b) \in {\mathbb S}^{2n+1}$. Then $\textnormal{deg}(\phi) =1$ when $n$ is odd, and $\textnormal{deg}(\phi) $ is odd when $n$ is even.
\end{theorem}

As a remark, Karpukhin and Stern \cite[Lemma 4.2]{KS20} showed that the degree is odd with a similar reflection symmetry condition.

The strategy of the proof is derived from an extension proof of the well-known Borsuk--Ulam theorem. See \cite[Theorem 5.2]{OR09}. The goal is to  extend the map on the sphere to the one on the closed ball while preserving continuity and the reflection symmetry. We focus on the fact that on the hyperplane $\R^{2n+1} = \{ (a,b): a \in \R^{n+1}, b \in \R^{n} \times \{ 0\}\}$, the reflection map $R_b$ can be written as a Cartesian product of the reflection on  $\R^{n}$ and an identity on $\R$. Considering this hyperplane allows us to extend the map step-by-step by induction on dimension of the domain. In particular, the dimension of the parameter of the reflection map determines the inductive step. The extension of the map to all of $\overline\B^{2n+2}$ consists of an identity map on an $\epsilon$-ball at the origin and a continuous map with reflection symmetry on the remaining domain. After proving these extension lemmas, we calculate the degree of the extended map using the reflection symmetry. 

For simplicity, $i$ is used for an inclusion map or identity map depending on the context. We use $\R^n$ instead of $\R^{2n+1}$ in the next lemma to keep the notation simple for now, and return to using $\R^{2n+1}$ again after the next lemma. 
\begin{defn}
Let $D \subset \R^n$ be a set and $k$ be a fixed integer such that $1 \leq k  \leq n/2$. Consider $(a,b) \in \R^{n}$ where $a \in \R^{n-k}$ and $b \in \R^k  $. The domain is called \textit{$k$-symmetric} if $(a,b) \in D$ $\iff$ $((R_b \times i) (a), -b) \in D$, whenever $b \neq 0$. In the special case when $n$ is even and $k=n/2$, $D \subset \R^n$ being \textit{$n/2$-symmetric} means $(a,b) \in D$ $\iff$ $(R_b (a), -b) \in D$.
\end{defn}

Let's first extend a map when the dimension of its $k$-symmetric domain is lower than that of the codomain.
\begin{lemma} \label{extend}
Let $D \subset \R^{n}$ ($n \geq 3$) be bounded, open and $k$-symmetric with fixed $1 \leq k < n/2$ such that $0 \notin \overline{D}$. Let $\phi: \partial D \to \R^m \setminus\{0\}$, $n<m$ ($m$ is even) be a continuous map with the following reflection symmetry property:
\begin{equation}
\label{refsym2}
\begin{split}
(R_b \times i \times R_b \times i) \phi (a,b) &= \phi((R_b \times i) (a), -b)  \qquad  \text{when $b \neq 0$,} \\
\phi(a,0) &= (a,0)  \in \R^m \qquad\qquad\quad \! \text{when $b = 0$,} 
\end{split}
\end{equation}
for all $(a,b) \in \partial D$ with $a \in \R^{n-k}$  and $b \in \R^k$. Then there exists a continuous map $\varphi: \overline{D} \to \R^m \setminus \{0\}$ which extends $\phi$ and satisfies the properties \autoref{refsym2} for $(a,b) \in \overline{D}$ .
\end{lemma}

Here are a few remarks about the statement to illustrate the lemma. The reason for assuming $n<m$ and $m$ is even is due to the definition of reflection symmetry property: on the left hand side of \autoref{refsym2}, $R_b$ acts on $k$ dimensional space and $i$ acts on $(m/2-k)$-dimensional space; on the right side, $i$ is an identity map acting on $(n-2k)$-dimensional space.
\begin{proof}
Define $\phi(a,0) =(a,0)$ for all $a \in \R^{n-k}$, noting this extended $\phi$ is still continuous. We begin induction on $n$ with the case when $n=3$. Here, we only consider the case when $k=1$, since $1 \leq k<3/2$: $a \in \R^2 $ and $b \in \R$. For convenience, we identify the plane $\R^{2}$ with the hyperplane $\R^2 \times \{0\}$ in $\R^3$. Let $\phi_1: \partial D \cup ( \overline{D} \cap \R^2)  \to \R^m$ be defined as
\begin{align*}
\phi_1 =
\begin{cases}
\phi \qquad  \text{on $\partial D$,}\\
i \qquad \,\, \text{on $\overline{D} \cap \R^2$}.
\end{cases}
\end{align*}
Note that the map is well-defined and continuous, satisfying the reflection symmetry. The image does not contain 0 since $\phi$ does not vanish on $\partial D$ and $0 \notin \overline{D}$. Let $D_+ :=\{(a,b) \in D : b>0 \}$ and $D_{-} :=\{(a,b) \in D : b<0 \}$. We can rewrite $\overline D$ as $$\overline{D} = \partial D \cup D_+ \cup D_- \cup (\overline{D} \cap \R^2).$$ By the continuous extension lemma \cite[Lemma 5.1 (1)]{OR09} and since $\partial D_{+}$ is a compact set, there exists a continuous extension of $\phi_1 |_{\partial{D_{+}}}$ to $\phi_2: \overline{D}_{+} \to \R^m$ that is nowhere zero. Using the fact that $R_b$ is simply the map $N: \R \to \R$, $N(x)=-x$, when $b \neq 0$, let $\varphi: \overline{D} \to \R^m$ be defined as
\begin{align*}
\varphi(a,b) =
\begin{cases}
\phi_2(a,b) \qquad \qquad\qquad\qquad\qquad\qquad \quad \quad\text{when $(a,b) \in \overline{D}_{+}$,}\\
\left( N \times i  \times N \times i \right) \phi_2 ((N \times i)(a), -b)\qquad \text{when $(a,b) \in \overline{D}\setminus\overline{D}_{+}$.}
\end{cases}
\end{align*}
One can check that the map is continuous, nowhere-zero and satisfies the reflection symmetry. 

Suppose for induction that the lemma holds for domains in $\R^{n-1}$, $n\geq 4$. We prove the result for $D \subset \R^n$. Consider $k$ to be any integer such that $1 \leq k < n/2 $. Let $a \in \R^{n-k}$ and $b \in \R^k$. Denote the last coordinate of $b$ as $b_k$ and identify $\R^{n-1}$ with $\{(a,b) \in \R^{n} : b_k =0\}$.

By the induction hypothesis, we may extend the map from $\partial D\cap\R^{n-1}$ to $ \overline{D} \cap \R^{n-1}$, as follows. Note that when $b_k =0$, the reflection map $R_b$ can be written as a Cartesian product of the reflection map and identity or inclusion map on the last coordinate. We first check the case when $k >1$. By the induction hypothesis with $k-1$ and $n-1$, we can extend the map on $\partial D\cap\R^{n-1} $ to the map on all of $\overline{D} \cap \R^{n-1}$ so that it is continuous, nowhere-zero and satisfies the reflection symmetry. When $k=1$, the map can simply be extended as $\phi(a,0) =(a,0) \in \R^m$ on $\overline{D} \cap \R^{n-1}$ since $b \in \R$. In either case, we denote the extension map of $\phi|_{\partial D\cap\R^{n-1}}$ as $\phi_1: \overline{D} \cap \R^{n-1} \to \R^m$. 

The extension from previous step allows us to define $\phi_2:  ( \overline{D} \cap \R^{n-1}) \cup \partial D \to \R^m$ as 
\begin{align*}
\phi_2 =
\begin{cases}
\phi \qquad \,\, \text{on $\partial D$,}\\
\phi_1 \qquad \text{on $\overline{D} \cap \R^{n-1}$.}
\end{cases}
\end{align*}
Let $D_+ :=\{(a,b) \in D : b_k>0 \}$ and $D_{-} :=\{(a,b) \in D : b_k<0 \}$. We can rewrite $\overline D$ as $\overline{D} = \partial D \cup D_+ \cup D_- \cup (\overline{D} \cap \R^{n-1}).$ Similarly, by the continuous extension lemma \cite[Lemma 5.1 (1)]{OR09}, there exists an extension of $\phi_2 |_{\partial{D_{+}}}$ to $\phi_3: \overline{D}_{+} \to \R^m$ that is nowhere-zero. Define $\varphi: \overline{D} \to \R^m$ as
\begin{align*}
\varphi(a,b) =
\begin{cases}
\phi_3(a,b) \qquad \,\,\, \qquad\qquad\qquad\qquad\qquad\quad\quad\,\text{when $(a,b) \in \overline{D}_{+}$,}\\
\left( R_b \times i  \times R_b \times i \right) \phi_3 ((R_b \times i)(a), -b)\qquad \text{when $(a,b) \in \overline{D}\setminus\overline{D}_{+}$.}
\end{cases}
\end{align*}
One can check that the map is nowhere-zero and satisfies the reflection symmetry, and is continuous where $b_k \neq 0$. Showing continuity where $b_k =0$ requires a careful argument. We only need to check the direction from $D_-$ since the other direction is easy due to the continuity of $\phi_3$. When $k=1$, the reflection map $R_b=N$ does not depend on $b$ and the continuity from $D_-$ follows easily. When $k \neq 1$, we want to show that as $(a,b) \to (\tilde{a},\tilde{b})$ where $b_k >0$ and $\tilde{b}_k =0$, we have $\left( R_b \times i  \times R_b \times i \right) \phi_3 ((R_b \times i)(a), -b) \to \phi_3(\tilde a,\tilde b)$. Note that when $\tilde b  \neq 0$, $R_b \to R_{\tilde b }$ continuously and so the norm of the difference
\begin{align*}
&|\left( R_b \times i  \times R_b \times i \right) \phi_3 ((R_b \times i)(a), -b) - \phi_3(\tilde a,\tilde b)|\\
&=|\phi_3 ((R_b \times i)(a), -b) -\left( R_b \times i  \times R_b \times i \right) \phi_3(\tilde a,\tilde b)| 
\end{align*}
approaches 0 as $(a,b) \to (\tilde{a},\tilde{b})$ by the reflection symmetry of $\phi_2$ on $\overline{D} \cap \R^{n-1}$. When $b=0$,  it is enough to consider a sequence $\{(a^j,b^j)\}_{j=1}^{\infty}$ in $D_-$ converging to $(\tilde{a},0)$. By the extension $\phi_3(\tilde a, 0)=(\tilde a, 0)$ for all $a \in \R^{n-k}$, we may write the norm of the difference as
\begin{align}
&|\left( R_{b^j} \times i  \times R_{b^j} \times i \right) \phi_3 ((R_{b^j} \times i)(a^j), -b^j) - (\tilde a, 0)| \nonumber
\\&=|\phi_3 ((R_{b^j} \times i)(a^j), -b^j) - (( R_{b^j} \times i) (\tilde a), 0) |\nonumber
\\& = |\phi_3 ((R_{b^j} \times i)(a^j), -b^j) - \phi_3  (( R_{b^j} \times i) (\tilde a), 0) |.\label{con1} 
\end{align}
We want to show that this quantity approaches 0. Note that  the last line holds even though $ (( R_{b^j} \times i) (\tilde a), 0)$ might not be in $D_+$, since we extended $\phi_3$ to all of $b \neq0$. Moreover, we have that the quantity
$$|((R_{b^j} \times i)(a^j), -b^j)-(( R_{b^j} \times i) (\tilde a), 0)|=|(a^j,b^j) -( \tilde a, 0)|$$ 
approaches 0 as $j \to \infty $. By the uniform continuity of $\phi_3$ on compact sets, the quantity \autoref{con1} approaches 0 as $j \to \infty$.
\end{proof}

The next lemma shows that the map can be extended from the boundary to the interior when the dimensions of the domain and the codomain are equal. For convenience, we identify $\R^{2n+1}$ with the hyperplane $\R^{2n+2} \cap \{ b_{n+1} =0\}$.
\begin{lemma} \label{extend2}
Let $D \subset \R^{2n+2}$ ($n \geq 1$) be bounded, open and $(n+1)$-symmetric set such that $0 \notin \overline{D}$. Let $\phi: \partial D \to \R^{2n+2} \setminus \{0\}$ be a continuous map with the reflection symmetry property \autoref{refsym2}. Then $\phi$ can be extended to $\varphi: \overline{D} \to \R^{2n+2}$ which is continuous with reflection symmetry and $\varphi \neq 0 $ on $\overline{D} \cap \R^{2n+1}$.
\end{lemma}

\begin{proof}
Let $D_+ :=\{(a,b) \in D : b_{n+1}>0 \}$ and  $D_{-} :=\{(a,b) \in D : b_{n+1}<0 \}$. By \autoref{extend}, extend $ \phi|_{\partial D\cap \R^{2n+1} }$ to $\phi_1:\overline{D}\cap \R^{2n+1}  \to \R^{2n+2}$ so that it is continuous, nowhere-zero and satisfies the reflection symmetry. Define $\phi_2: (\overline{D}\cap \R^{2n+1} ) \cup \partial D \to \R^{2n+2}$ by 
\begin{align*}
\phi_2 =
\begin{cases}
\phi \qquad \,\, \text{on $\partial D$,}\\
\phi_1 \qquad \text{on $\overline{D} \cap \R^{2n+1}$.}
\end{cases}
\end{align*}
Let $\widehat{D} = \partial D \cup (\overline{D}\cap \R^{2n+1} ) \cup \overline{D}_{+}$. Next, we use the Tietze extension theorem to extend the map $\phi_2|_{\partial D \cup (\overline{D}\cap \R^{2n+1})}$ to $\phi_3: \widehat{D} \to \R^{2n+1}$ since $\partial D_+ \subset \widehat{D}$ are both compact sets. As a remark, the extension theorem used here is different from the one used earlier since the dimensions are now the same between the domain and the codomain, so the nowhere-zero property only holds at $\overline{D}\cap \R^{2n+1}  $ instead of $\overline{D}$. Let $\varphi: \overline{D} \to \R^{2n+2}$ be defined as
\begin{align*}
\varphi(a,b) =
\begin{cases}
\phi_3(a,b) \qquad\qquad\qquad\qquad \quad \text{when $(a,b) \in \widehat{D}$,}\\
(R_b \times R_b)\phi_3(R_b(a),-b) \qquad \text{when $(a,b) \in D_-$.}
\end{cases}
\end{align*}
It is easy to check that the map $\varphi$ is nowhere-zero in $\overline{D}\cap \R^{2n+1} $ and satisfies the reflection symmetry. Continuity follows by a similar argument to the previous lemma.
\end{proof}
Next, for \autoref{changedeg} we rely on the following proposition from \cite[Proposition 3.2]{OR09}. 
\begin{proposition} 
\label{degdiff}
Let $f: \overline{D} \to \R^{n+1}$ be a continuous mapping whose restriction to $D$ is $C^1$ and let $a \in \R^{n+1} \setminus f(\partial D)$ be a regular value of $f|_{D}$. Then $f^{-1}(a)$ is finite and 
\begin{align*}
d(f,D,a) = \displaystyle\sum_{x \in f^{-1}(a)} \textnormal{signdet} \left(J_f(x)\right).
\end{align*}
\end{proposition}
For the next lemma, we use the same notations $D_{+}$ and $D_{-}$ for the sets with a certain reflection relationship (not necessarily depending the sign of the last coordinate). They are more general sets compared to the ones in the previous extension lemmas. 

\begin{lemma}[Change of Variables of the Degree]
\label{changedeg}
Let $D_{+} \subset \R^{2n+2}$ be a bounded and open set such that for $(a,b) \in \R^{n+1}$, we have $b \neq 0$ whenever $(a,b) \in  D_{+}$. Define $D_{-}:= \{ (R_b(a),  -b ): (a,b) \in D_{+}\}$ and let $\phi: \overline{D}_{-} \to \R^{2n+2}$ be a continuous map with $\phi(a,0) =(a,0)$ whenever $(a,0) \in \overline{D}_{-}$, and $0 \notin \phi( \partial D_{-}) $. If the map $\Psi: \overline{D}_+ \to \R^{2n+2}$ is defined as
\begin{align*}
\Psi(a,b) &=  (R_b \times R_b) \phi(R_b(a),-b), \qquad\text{for } b \neq 0,  \\
\Psi(a,0) &= (a,0), \qquad\qquad\qquad\quad\quad\quad \,\,\text{for $b= 0$}, 
\end{align*}
then the degrees are related by $d(\phi, D_{-}, 0)=(-1)^{n} d(\Psi, D_{+}, 0)$.
\end{lemma}
Note that whenever $b=0$ on $\overline{D}_+$, the point lies on the boundary $\partial D_+$, similarly for $\overline{D}_-$. Moreover, $0 \notin \Psi( \partial D_{+}) $ due to the assumption $0 \notin \phi( \partial D_{-}) $. 

\begin{proof}
\textit{Step 1 --  Continuity of the maps.} 

To show that $\Psi$ is continuous on $\overline{D}_{+}$, we only need to check boundary points where $b=0$, due to continuity of $\phi$ and the assumption that $b \neq 0$ whenever $(a,b) \in  D_{+}$. The continuity of $\Psi$ follows by arguing like in the proof of \autoref{extend}. 

For the next step, define $\psi: D_{+} \to \R^{2n+2}$ as
\begin{align*}
\psi(a,b) &= \phi(R_b(a),-b)
\end{align*}
As a remark, $\psi$ might not be extended to the closure but is introduced for convenience. We write $\phi(a,b) = (\phi_1(a,b), \phi_2(a,b))$ and similarly for $\psi$ and $\Psi$. 

\textit{Step 2 -- The degree of $\Psi$ in terms of $\psi$.} 

We first assume that $\phi \in C(\overline{D}_{-}) \cap C^1(D_{-})$ and that $0$ is a regular value. By \autoref{degdiff} applied to $\Psi$, and using that $b \neq 0$ on $D_+$,
\begin{align}
d(\Psi,D_+,0) &= \displaystyle\sum_{(a,b) \in \Psi^{-1}(0)} \text{sign} \text{det} \left( J_{\Psi}(a,b) \right) \nonumber\\
&= \displaystyle\sum_{(R_b \times R_b)\psi(a,b)=0}\text{sign}  \text{det}  
\begin{pmatrix}
\partial_{a} R_b \psi_1(a,b) & \partial_{b} R_b \psi_1(a,b) \\[12pt]
\partial_{a} R_b \psi_2(a,b) &\partial_{b} R_b \psi_2(a,b) 
\end{pmatrix}\nonumber \\
& = \displaystyle\sum_{\psi(a,b)=0}\text{sign}  \text{det}  
\begin{pmatrix}
R_b (\partial_{a}  \psi_1(a,b)) &R_b ( \partial_{b}  \psi_1(a,b)) \\[12pt]
R_b (\partial_{a} \psi_2(a,b)) &R_b (\partial_{b} \psi_2(a,b))
\end{pmatrix} \label{productrule},
\end{align}
as follows. In the first column of the matrix in \autoref{productrule}, the reflection map $R_b$ is independent of the variable $a$. For the second column, the terms including $\psi(a,b)$ vanish because we evaluate the derivatives at $\psi(a,b)=0$; this allows us to take $R_b$ outside the partial derivatives.

\textit{Step 3 -- The degree of $\psi$ in terms of $\phi$.} 

We calculate the sign of the determinants inside the sum in \autoref{productrule}, 
\begin{align}
&\text{sign}\text{det}  
\begin{pmatrix}
R_b (\partial_{a}  \psi_1(a,b)) &R_b ( \partial_{b}  \psi_1(a,b)) \\[12pt]
R_b (\partial_{a} \psi_2(a,b)) &R_b (\partial_{b} \psi_2(a,b))
\end{pmatrix} \nonumber \\
&=\text{sign}  \text{det}  
\begin{pmatrix}
R_b & 0 \\[12pt]
0& R_b
\end{pmatrix}
\begin{pmatrix}
\partial_{a}\psi_1(a,b) & \partial_{b} \psi_1(a,b)\\[12pt]
\partial_{a}\psi_2(a,b) &\partial_{b}\psi_2(a,b)
\end{pmatrix} \nonumber \\
& = \text{ sign} \text{det} 
J_{\psi}(a,b) ,
\label{degpsi}
\end{align}
where $R_b$ is the matrix for the reflection map $R_b$ and the determinant of the block matrix is $1$ since det$R_b = -1$. Next, we calculate the Jacobian $J_{\psi}(a,b)$ in \autoref{degpsi}:

\begin{align}
J_{\psi}(a,b) &= 
\begin{pmatrix}
\partial_{a} \left( \phi_1(R_b(a),-b) \right)& \partial_{b} \left( \phi_1(R_b(a),-b)\right) \\[12pt]
\partial_{a} \left( \phi_2(R_b(a),-b)\right) &\partial_{b} \left( \phi_2(R_b(a),-b)\right)
\end{pmatrix}
 \nonumber \\ 
&= 
J_{\phi}(R_b(a),-b)
\begin{pmatrix}
 R_b & \partial_{b}  R_b(a) \\[12pt]
0& -I
\end{pmatrix}  \label{degpsi2},
\end{align}
by the chain rule. The determinant of the second matrix in \autoref{degpsi2} is $(-1)^{n+2}$ since the determinant of the block matrix can be calculated as for a 2$\times$2 matrix when the lower left submatrix is 0. So we can conclude that $d(\phi, D_{-}, 0)=(-1)^{n} d(\Psi, D_{+}, 0)$.

\textit{Step 4 -- Smooth approximation of continuous $\phi$.}

We now assume that $\phi$ is just continuous on $\overline{D}_{-}$ and 0 is not necessarily a regular value. 

First suppose there is no point of the form $(a,0)$ in $\overline{D}_-$. That is, if $(a,b) \in  \overline{D}_-$ then $b \neq 0$. Choose a $C^{1}(\overline{D}_{-})$-map $\tilde{\phi}$ homotopic to $\phi$ such that 0 is a regular value with $0 \notin \tilde{\phi}(\partial D_{-})$. To recall briefly about the existence of such functions, we use the extension theorem to extend $\phi$ outside the domain $D_-$ and then we use mollification. We can choose $\tilde{\phi}$ close enough to $\phi$ so that the image of the boundary does not contain 0. Also, if $0$ is not a regular value, by Sard's theorem, there exists a close enough value $p$ near 0 that is a regular value. We shift the map by $-p$ and hence make 0 a regular value of $\tilde{\phi}$. Due to the assumption that there is no point of the form $(a,0)$ in $\overline{D}_-$, the condition $\phi(a,0) =(a,0)$ does not need to be considered for mollification. 

Define $\tilde{\Psi}:D_{+} \to \R^{2n+2}$ from $\tilde{\phi}$ similarly to how $\Psi$ is constructed from $\phi$. Note that this map is $C^{1}(D_{+})$ and continuous on $\overline{D}_{+}$. We conclude that  $d(\tilde{\phi}, D_{-}, 0)=(-1)^{n} d(\tilde{\Psi}, D_{+}, 0)$ by repeating step 2 and 3.

Next, based on the fact that  $R_b$ is a linear isomorphism when $b \neq 0$, we show that $\tilde{\Psi}$ is homotopic to $\Psi$. Consider a  continuous homotopy $h_t(a,b)$ defined on $\overline{D}_{-} \times [0,1]$ where $h_0(a,b)= \tilde{\phi}$ and $h_1(a,b)= \phi$. We may choose that $0 \notin h_t( \partial D_-)$ for all $t \in [0,1]$.

Define $k_t(a,b):= (R_b \times R_b) h_t(R_b(a),-b)$. We want to show that $k_t$ is a homotopy from $\tilde{\Psi}$ to $\Psi$ and $0 \notin k_t(\partial D_{+})$ for all $t \in [0,1]$. The continuity of $k_t$ follows since $h_t$ is continuous and $R_b \times R_b$ is continuous when $b \neq 0$. We show $0 \notin k_t(\partial D_{+})$ by contradiction. Suppose $0 \in  k_t(\partial D_{+})$, say $k_t(a,b) =0$ for some $(a,b) \in \partial D_+$. Then $h_t(R_b(a),-b) =0$, which is impossible since $(R_b(a),-b) \in \partial D_-$. Hence, $0 \notin k_t( \partial D_+)$. We have $d(\phi, D_{-}, 0)= d(\tilde{\phi}, D_{+}, 0)$ and similarly for $\Psi$ and $\tilde{\Psi}$, so we have the conclusion in Step 3.

Next suppose there does exist a point of the form $(a,0)$ in $\overline{D}_-$, define a set $D_{-, \epsilon} \subset D_-$ where $|b| >\epsilon$ for all $D_{-, \epsilon}$ such that $ 0 \notin \phi(D_- \setminus D_{-, \epsilon} )$. Such an $\epsilon$ exists since $\phi$ is continuous and $0 \notin \phi( \partial D_{-}) $. By the excision property \cite[p. 44]{OR09}, $d(\phi|_{D_{-,\epsilon}},D_{-, \epsilon},0) =d(\phi, D_{-}, 0)$ and $d(\Psi|_{D_{+, \epsilon}},D_{+, \epsilon},0) =d(\Psi, D_{+}, 0)$. Note that $ D_{-, \epsilon} $ was explained above as the case when $b \neq 0$. By combining the degree relationship between $D_{-}$ and $D_{-,\epsilon}$, and the previous case, we have the conclusion of the lemma.
\end{proof}

As a remark, we are not assuming any reflection symmetry for $\phi$ in the previous lemma since we only want to understand the relationship of the degrees between $\phi$ and $\Psi$. 
\begin{proof}[Proof of \autoref{topdeg}]
Fix $0< \epsilon <1$ and define $U := \overline{B(0, \epsilon)} \subset \B^{2n+2} $ and $D_1 := \B^{2n+2} \setminus U$.

Define $\phi_1: U \cup {\mathbb S}^{2n+1} \to \R^{2n+2}$ by 
\begin{align*}
\phi_1 =
\begin{cases}
i \quad\quad\, \text{on $ U$},\\
\phi \qquad \text{on ${\mathbb S}^{2n+1}$}.
\end{cases}
\end{align*}
Then $\phi_1$ is a continuous function on $\partial D_1 = {\mathbb S}^{2n+1} \cup \partial U$ that is the identity near the origin and $\phi_1$ satisfies the reflection symmetry. By \autoref{extend2}, there exists a continuous extension $\phi_2$ of $\phi_1 |_{\partial D_1} $ such that $\phi_2:\overline{D}_1 \to \R^{2n+2}$ with the reflection symmetry and $\phi_2 \neq 0$ on $\overline{D}_1 \cap \R^{2n+1}$. 

Define the following continuous map $\phi_3$ on $\B^{2n+2}  =U \cup \overline{D}_1$,
\begin{align*}
\phi_3=
\begin{cases}
i \qquad \,\,\,\, \text{on $ U$,}\\
\phi_2 \qquad \text{on $\overline{D}_1$}.
\end{cases}
\end{align*}
It is easy to check that the map is well-defined, continuous, and satisfies the reflection symmetry.  

We first show that $d(\phi_2, D_1, 0)$ is 0 when $n$ is odd and even when $n$ is even. Let $D_1^+ :=\{(a,b) \in D_1 : b_{n+1}>0 \}$ and $D_1^{-} :=\{(a,b) \in D_1 : b_{n+1}<0 \}$. Using the additivity property \cite[Proposition 2.5]{OR09}, we deduce
$$
d(\phi_2, D_1,0) =  d(\phi_2,D^{+}_1, 0 )+ d(\phi_2,D^{-}_1, 0 ),
$$
since $0 \notin \phi_2(\overline{D}_1 \setminus (D^{+}_1 \cup D^{-}_1)$.
Next, using \autoref{changedeg}, we conclude 
\begin{align*}
d(\phi_2, D_1,0)  =
\begin{cases}
0 \qquad\qquad\qquad\quad\,  \text{when $n$ is odd,}\\
2d(\phi_2,D^{+}_1, 0 )\qquad \text{when $n$ is even.}
\end{cases}
\end{align*}
Denote int as the interior of a set. By applying the additivity property again, we have
\begin{align*}
\text{deg}(\phi) &= d(\phi_3, \B^{2n+2},0) \\
&= d(\phi_3, \B^{2n+2} \setminus U,0) +d(\phi_3, \text{int}\, U,0)\\
&= d(\phi_2, D_1,0)+d(i, \text{int}\,U,0) \\
&= d(\phi_2, D_1,0)+1,
\end{align*}
where the additive property holds since $0 \notin \phi_3 (\overline\B^{2n+2} \setminus ( D_1 \cup \text{int}\,U))$. Hence, $\text{deg}(\phi) =1$ when $n$ is odd and $\text{deg}(\phi)$ is odd when $n$ is even, which finishes the proof of \autoref{topdeg}.
\end{proof}

\section{\bf Estimating the Rayleigh Quotient --- proof of \autoref{thm}} \label{sec:rayleigh}
Recall that the Veronese map $\Phi$ maps $\RP^n$ into ${\mathbb S}^{m-1}$ where $m(n)$ is written as $m$ for simplicity. We apply the trial functions constructed in \autoref{sec:trial} to the Rayleigh quotient, to estimate the second eigenvalue from above.

The variational characterization states that
\begin{align}
\label{Rayleigh}
\lambda_{2}(\RP^{n},w) \leq \frac{\int_{\RP^n}^{}|\nabla_{\!wg\,}u|_{wg}^{2} \, dv_{wg}}{\int_{\RP^n}^{}u^{2} \, dv_{wg}},
\end{align}
where $u$ is a trial function in $H^{1}(\RP^n)$ which satisfies the orthogonality condition to the constant and to the first excited state. By substituting the definition of trial functions \autoref{vec} into \autoref{Rayleigh}, we have

\begin{equation}
\label{Rayleigh2}
\lambda_{2}(\RP^{n}, w) \leq \frac{\int_{\RP^{n}}^{}|\nabla_{\!wg\,}(T_{-c} \circ F_{H} \circ \Phi)_{j}|_{wg}^{2} \, dv_{wg}}{\int_{\RP^{n}}^{}(T_{-c} \circ F_{H} \circ \Phi)_{j}^{2} \, dv_{wg}}, \qquad j=1, \dots , m.
\end{equation} 

For each $j$, we multiply by the denominator on both sides of the inequality \autoref{Rayleigh2} and sum for all $j$'s. Note that $(T_{-c} \circ F_{H} \circ \Phi)(y)$ is a point on the unit sphere. This implies that the sum of the denominators is equal to the volume of $\RP^n$ with the metric $wg$, that is, $ \int_{\RP^n} \sum_{j} (T_{-c} \circ F_{H} \circ \Phi)_j^{2} dv_{wg} = \V (\RP^n, wg)$. Recall that by our normalization, $\V (\RP^n, wg)= \V_n (1)$. Hence,  the inequality in \autoref{Rayleigh2} can be written as,
\[
\lambda_{2}(\RP^{n},w) \V_n(1) \leq  \int_{\RP^{n}}^{} \sum_{j=1}^{m(n)}|\nabla_{\!wg\,}(T_{-c} \circ F_{H} \circ \Phi)_{j}|_{wg}^{2} \, dv_{wg}.
\]
By applying H\"{o}lder's inequality to the right hand side,
\begin{align*}
 &\lambda_{2}(\RP^{n},w) \V_n(1)  
\\ &\leq \left( \int_{\RP^{n}}^{}  \left( \sum_{j=1}^{m(n)} |\nabla_{\! wg\,}(T_{-c} \circ F_{H} \circ \Phi )_{j}|_{wg}^{2}\right)^{\!\! n/2} dv_{wg}  \right)^{ \! \!2/n} \! \V_n(1)^{1-2/n}.
\end{align*}
(As a remark, the application of H\"{o}lder is not needed when $n=2$.) We introduce a new notation: for an $n$-dimensional manifold $(M,g)$ and a map $F:M \to {\mathbb S}^{m-1}$, we denote $|\nabla_g F|_{g} = \sqrt{\sum_{j=1}^{m} |\nabla_{\!g\,}(F )_{j}|_{g}^{2}}$. The inequality becomes
\begin{align}
 \lambda_{2}(\RP^{n},w)^{n/2}  \, \V_n(1)  &\leq   \int_{\RP^{n}}^{}|\nabla_{\!wg\,} (T_{-c} \circ F_{H} \circ \Phi)|_{wg}^{n} \, dv_{wg} \nonumber
\\& = \int_{\RP^{n}}^{}|\nabla_{\!g \,}(T_{-c} \circ F_{H} \circ \Phi)|_{g}^{n} \, dv_{g}, \nonumber
\end{align}
by changing $wg$ to the round metric $g$ on $\RP^n$ and using conformal invariance. We again change the variables from $\RP^n$ to $\Phi(\RP^n)$, which is an embedded submanifold of $\mathbb S^{m-1}$. Denote $n$-dimensional Hausdorff measure on the submanifold as $d{\mathcal H}_n$ and write $\nabla$ for the gradient on the embedded submanifold. Since $\Phi$ is conformal by \autoref{phiconformal},
\begin{align}
\lambda_{2}(\RP^{n},w)^{n/2}  \, \V_n(1) &  \leq \int_{\Phi(\RP^{n})}^{}|\nabla(T_{-c} \circ F_{H} )|^{n} \, d{\mathcal H}_n \nonumber
\\&=   \int_{\Phi(\RP^n) \cap  H }^{}|\nabla(T_{-c})|^{n} \, d{\mathcal H}_n+ \int_{\Phi(\RP^n) \cap H^c }^{}|\nabla (T_{-c} \circ R_H)|^{n} \, d{\mathcal H}_n, \nonumber
\end{align}
where $H^c= {\mathbb S}^{m-1} \setminus H$. 

By changing variable in each part and using the fact that the reflection $R_H$ and M\"{o}bius transformation $T_{-c}$ are conformal, we have the following inequality:
\begin{align}
 &\lambda_{2}(\RP^{n},w)^{n/2} \, \V_n(1)  \nonumber
\\&\leq \int_{T_{-c}(\Phi(\RP^n) \cap H )  }^{}|\nabla y|^{n} \,d{\mathcal H}_n (y)  +\int_{ T_{-c}( R_{H}(\Phi(\RP^n)) \cap H^{\text{int}})}|\nabla y|^{n} \,d{\mathcal H}_n (y) \label{applychange}
\end{align}
where $H^{\text{int}} =R_H (H^{c})$ is the interior of $H$. Recall that the Veronese surface is embedded into ${\mathbb S}^{m-1}$. At each point $\tilde{y}$ in the surface, the tangent space has dimension $n$, and after rotating the coordinate system, we may suppose $\tilde{y}=(0, \dots, 0, 1)$ and the tangent space can be written as the span of $\{ \partial_{y_1}, \dots ,\partial_{y_n} \}$ where $y = (y_1, \dots, y_n, y_{n+1}, \dots, y_m)$. At $\tilde{y}$ we compute
\begin{align}
|\nabla y|^n= \left( \sum^{n}_{j=1} |\nabla y_j|^2 \right)^{\!n/2}= n^{n/2}. \label{eigenvalue}
\end{align}
Now substitute \autoref{eigenvalue} into \autoref{applychange}. Then we have
\begin{align}
 &\lambda_{2}(\RP^{n},w)^{n/2} \, \V_n(1)  \nonumber
\\&\leq n^{n/2}\big( \V_n \left( T_{-c}(\Phi(\RP^n) \cap H)  \right) + \V_n \left( T_{-c}( R_{H}(\Phi(\RP^n)) \cap H^{\text{int}}\right) \big) \nonumber
\\&< n^{n/2}\big( \V_n ( T_{-c}(\Phi(\RP^n))  ) + \V_n ( T_{-c}( R_{H}(\Phi(\RP^n)) ) \big), \label{volumedivide}
\end{align}
where we dropped the intersections with $H$ and $H^{\text{int}}$. Note that the last inequality \autoref{volumedivide} is strict, as follows. Either $\Phi(\RP^n) \cap H^c$ has positive $n$-volume or else $\Phi(\RP^n) \cap H$ does, or both, and so after reflecting the second set with $R_H$ we deduce that $\Phi(\RP^n) \cap H^c$ or $R_H(\Phi(\RP^n)) \cap (H^{int})^c$ or both have positive $n$-volume. Hence when we drop the intersections with $H$ and $H^{int}$ in \autoref{volumedivide}, at least one of the volumes becomes strictly larger. This last line is bounded above by $2 n^{n/2}$ times the $(m-1)$-conformal volume of $\Phi$ (see El Soufi and Illias \cite[p.259]{EI86}), since $T_{-c}$ and $T_{-c} \circ R_H$ are conformal diffeomorphisms of the sphere $\mathbb{S}^{m-1}$. The $(m-1)$-conformal volume of $\Phi$ equals $(a_n)^n \V_n (1) $ by \cite[Corollary 2.3 and p.267]{EI86}, and so
\begin{align}
 \lambda_{2}(\RP^{n},w)^{n/2} \,  \V_n(1)  &< 2n^{n/2} (a_n)^n \V_n (1) \nonumber\\ & =2 \V_n(1) \left( 2n+2 \right)^{\! n/2},   \nonumber
\end{align}
which finishes the proof of \autoref{thm}. 

\section{\bf Remark about \autoref{conj}}

The conjecture is based on the hope that in \autoref{volumedivide}, the sum of the volumes is maximized when $|c|<1$, $p \in \Phi (\RP^n)$, and $t \to 1$. In that limit, $\Phi (\RP^n) \cap H$ (non-reflected set) tends to the image of $n$-projective space under the Veronese map and $R_H(\Phi(\RP^n)) \cap H^{\text{int}}$ (reflected set) tends to an $n$-sphere embedded in the higher dimensional sphere, as explained below. This kind of “bubbling” is motivated by the proof of the conjecture in 2-dimension by Nadirashvili and Penskoi \cite{NA18}. Also, Petrides \cite{P22} proved that under certain Palais--Smale condition, the $k$-th eigenvalue for the compact $n$-dimensional manifold is maximal with a sequence of metrics that converges to that of “bubbling”. 

We have not been able to show that the sum of the volumes in \autoref{volumedivide} is maximal in the bubbling situation. In this section, though, we estimate the volume when the fold map degenerates ($t \to 1$) and separately when the M\"{o}bius transformation degenerates ($|c| \to 1$). Both cases support our conjecture.

First, we calculate that as the fold map $F_{p,t}$ degenerates when $t \to 1$, the volume approaches at most the volume of the projective space and a sphere. 
	\begin{proposition}[Fold map degenerates] Given an $N$-dimensional imbedded smooth surface $\gamma: \overline{\B}^N \to {\mathbb S}^{M}$, $1 \leq N <M$, we have
		\begin{equation}\label{folddegen}
\limsup_{t \to 1} {\mathcal H}_N (F_{p,t} \circ \gamma (\overline{\B}^N) ) \leq {\mathcal H}_N({\mathbb S}^N) +{\mathcal H}_N(\gamma(\overline{\B}^N)),
		\end{equation}
where $p \in {\mathbb S}^M$ is fixed.
	\end{proposition}
Here, we use the closed ball $\overline{\B}^N $ for its compactness, and the fold map $F_{p,t}$ acts on $\mathbb{S}^M$ (that is, $M$ is the “$m-1$” in \autoref{maps}).
\begin{proof}
Our main strategy is to decompose the image $\gamma (\overline{\B}^N)$ into two pieces, inside the cap $H_{p,t}$ and outside. Inside the cap, the fold map is the identity and so that part of $F_{p,t} \circ \gamma (\overline{\B}^N) $ has volume at most ${\mathcal H}_N(\gamma(\overline{\B}^N)) $. We will show that the fold of the piece outside has volume at most ${\mathcal H}_N({\mathbb S}^N)$ as $t \to 1$. For the rest of the proof, we may suppose the image $\gamma(\overline{\B}^N)$ contains $p$, since otherwise the image will lie entirely inside the cap when $t$ is close to 1 and there is nothing to prove.

Let us consider when the image $F_{p,t} \circ \gamma (\overline{\B}^N) $ is mapped onto the Euclidean space $\R^M$ under the stereographic projection $\Pi$, with $p$ mapping to the origin and $\partial H_{p,t}$ mapping to the sphere of radius $\epsilon$ centered at the origin, where $\epsilon \to 0$ as $t \to 1$. Write $\rho := \Pi \circ \gamma  : \overline{\B}^N \to {\R}^{M}$. After some appropriate  reparametrization of $\gamma$, we may assume that $\rho(0) = 0$ and its Jacobian $D\rho(0) =
\begin{pmatrix} 
I&0
\end{pmatrix}^T$, which is an $M \times N$ matrix.

Write $\B^M( \epsilon)$ as the image of the smaller spherical cap (the piece outside $H_{p,t}$) projected to the Euclidean space $\R^M$. The part of the surface intersecting the $\epsilon$-ball is written as $\rho: \Omega_{\epsilon} \to \B^M(\epsilon)$ where $\Omega_{\epsilon} \subset \overline{\B}^N$. 

Let the smooth surface $R_{\epsilon} \circ \rho: \Omega_\epsilon \to \R^M \setminus \B^M(\epsilon)$ be the fold of $\rho$ under the map $R_\epsilon :\R^M \to \R^M$ that reflects (inverts) the points in $\R^M$ across the boundary of $\B^M(\epsilon)$. It can be written as,
\[
R_{\epsilon} \circ  \rho (z) = \frac{\epsilon^2}{|\rho(z)|^2}\rho(z) 
\]
where $z \in \Omega_{\epsilon}$. Note that $R_\epsilon \circ \rho  =\Pi \circ F_{p,t} \circ \gamma $ on $\Omega_\epsilon$ since $R_\epsilon \circ \Pi = \Pi \circ R_{p,t}$. We want to estimate the limit of the weighted volume of $R_{\epsilon} \circ \rho(\Omega_\epsilon)$ from above,
\begin{align}
\label{volGam}
\limsup_{\epsilon \to 0}\int_{R_{\epsilon} \circ \rho(\Omega_\epsilon)}\left( \frac{2}{1+|z|^2} \right)^N  \, d \mathcal{H}_N \leq \mathcal{H}_N (\mathbb S^N),
\end{align}
where $\mathcal{H}_N$ is a Hausdorff measure for $N$-dimensional surface and $ 2/(1+|z|^2)$ is the conformal factor for stereographic projection.

The idea of the following proof is that the smooth surface near 0 can be approximated as a flat $N$-plane and so $R_\epsilon$ applied to that surface is approximately an $N$-plane outside the ball of radius $\epsilon$. 

Write $R=R_1$ for inversion in the unit sphere, so that $R_\epsilon = \epsilon^2 R$ and $R(z)$ has a conformal factor of $1/|z|^2$. Hence, the integral \autoref{volGam} becomes,
\begin{align}
&\int_{\rho(\Omega_\epsilon)}\left( \frac{2}{1+|\epsilon^2 R(z)|^2} \right)^N \left(\frac{\epsilon}{|z|} \right)^{2N} d \mathcal{H}_N \nonumber
\\&= \int_{\Omega_{\epsilon}}\left( \frac{2}{1+|\epsilon^2 R( \rho(z))|^2} \right)^N  \left(\frac{\epsilon}{|\rho(z)|} \right)^{2N}|D \rho(z)|  \, d \mathcal{H}_N  \nonumber \\
&=\int_{\Omega_{\epsilon}} \left( \frac{2\epsilon^2/|\rho(z)|^2}{1+\epsilon^4/|\rho(z)|^2}\right)^N |D \rho(z)| \, d \mathcal{H}_N , \label{volOme}
\end{align}
where $|D \rho(z)|:\R^N \to \R$ is the Jacobian determinant of $\rho$. By a change of variable with $z = \epsilon^2 \zeta$, the integrand in \autoref{volOme} becomes
\begin{align}
\left( \frac{2\epsilon^2/|\rho(\epsilon^2 \zeta)|^2}{1+\epsilon^4/|\rho(\epsilon^2 \zeta)|^2}\right)^N |D \rho(\epsilon^2 \zeta)| \epsilon^{2N} =  \left( \frac{2\epsilon^4}{|\rho(\epsilon^2 \zeta)|^2+\epsilon^4}\right)^N |D \rho(\epsilon^2 \zeta)|. \label{volOmeeps}
\end{align}
Note the domain $\Omega_\epsilon/ \epsilon^2$ approaches all of $\R^N$ when $\epsilon \to 0$ since the points in $\Omega_\epsilon$ have norm at most $O(\epsilon)$. The linear approximation of $\rho(z) \approx (z,0)$ at 0 implies that there exists an lower bound such that $|\rho(z)|^2 \geq |z|^2/2$ for small enough $\epsilon$, with $z \in \Omega_\epsilon$. Hence we can find an integrable dominator on $\R^N$ for the right hand side of \autoref{volOmeeps} as,
\begin{align}
1_{\Omega_\epsilon/ \epsilon^2}(\zeta) \left( \frac{2\epsilon^4}{|\rho(\epsilon^2 \zeta)|^2+\epsilon^4}\right)^N|D \rho(\epsilon^2 \zeta)|  &\leq  \left( \frac{2\epsilon^4}{ |\epsilon^2 \zeta|^2/2 +\epsilon^4}\right)^N 2  \nonumber \\
&\leq  \left( \frac{2}{ |\zeta|^2/2 +1}\right)^N 2 \qquad \text{where $\zeta \in \R^{N}$}. \label{volupper}
\end{align}
By using the upper bound \autoref{volupper}, we apply the dominated convergence to \autoref{volOmeeps} as $\epsilon \to 0$. To find the limit of the integrand, we expand with the Taylor expansion $\rho(\epsilon^2 \zeta) = \epsilon^2 \zeta + O(\epsilon^4)$ when $\epsilon$ is close enough to 0 and $\zeta$ is fixed. Then we evaluate the limit \autoref{volOmeeps} as $\epsilon \to 0$ as,
\begin{align}
 \int_{\R^N}\left( \frac{2}{|\zeta|^2 +1} \right)^N   d \mathcal{H}_N \label{volS}
\end{align}
by dominated convergence. The last integral is equal to $\mathcal{H}_N ({\mathbb S}^N)$. 

Note that $\R^N \setminus \Omega_\epsilon$ approaches to all of $\R^N$ as $\epsilon \to 0$ by the weak convergence of the measure in \autoref{sec: twostep}. The volume of this part becomes $\mathcal{H}_N(\gamma(\R^N))$, and added with \autoref{volS}, we have \autoref{folddegen} as $\epsilon \to 0$.
\end{proof}

Next we show that as the M\"{o}bius transformation degenerates, the volume of the image of the surface is at most that of ${\mathbb S}^m$. Recall that any point $y \in {\mathbb S}^M$ under the M\"{o}bius transformation $T_x(y)$ approaches $\tilde{x}$ as $x \to  \tilde{x}$ with $|\tilde{x}|=1$, except at $\tilde{y} =-\tilde{x}$. This implies that for any $N $-dimensonal submanifold where $N \leq M$ in ${\mathbb S}^M$,  the image under $T_x$ collapses to a point as $x$ approaches to the boundary unless the submanifold contains $\tilde y$, in which case the image stretches out to an $N$-sphere as $x \to \tilde{x}$.  

	\begin{proposition}[M\"{o}bius transformation degenerates] Given an $N$-dimensional embedded smooth surface $\omega: \overline{\B}^N \to {\mathbb S}^{M}$, $ 1 \leq N \leq M $, we have the following upper bound, 
		\begin{equation}
\limsup\limits_{|x|\rightarrow 1} \mathcal{H}_N(T_{x} \circ \,\omega(\overline{\B}^N)) \leq \mathcal{H}_N({\mathbb S}^N). \label{txdeg}
		\end{equation}
	\end{proposition}

	\begin{proof}
We write the south pole of ${\mathbb S}^M$ as $\mathbf{s}$ and its north pole as $\mathbf{n}$. Without loss of generality, we may assume that $x \to \mathbf{n}$. Then first assume that $\mathbf{s} \in \omega(\B^N)$.  We deal with the case when $\mathbf{s} \notin \omega(\B^N)$ later. We may parametrize the surface so that $\omega(0) = \mathbf{s}$ and $D\omega(0) =
\begin{pmatrix} 
I&0
\end{pmatrix}^T$. The left side of \autoref{txdeg} becomes 
		\begin{align}
\int_{T_{x} \circ \, \omega(\B^N)}  d \mathcal{H}_N&= \int_{\omega(\B^N)}\left(\frac{1-|x|^2}{1+|x|^2+2 x \cdot s}\right)^{\! N}  d \mathcal{H}_N(s)\qquad \text{(see \cite[2.1.7]{S16} and \cite[(30)]{A89}})\nonumber 
\\ &=\int_{\B^N}\left(\frac{1-|x|^2}{1+|x|^2+2x \cdot \omega(z)}\right)^{\! N}|D\omega(z)| \,  d \mathcal{H}_N(z)\label{intmob},
		\end{align}
where $|D\omega(z)|$ is the Jabobian determinant of $\omega$. Note that the integrand is similar to the Poisson kernel. We want to show that the integral in \autoref{intmob} is bounded above by the volume of the sphere as $x \to \mathbf{n}$. Note that the factor in \autoref{intmob} can be rewritten as 
$$ \frac{1-r^2}{1+r^2-2r \cos\theta(y, \omega(z))}
$$ in terms of $r =|x|$, $y = -x/|x|$, and the angle $\theta$ between $y$ and $\omega(z)$. Since $x \to \mathbf{n}$, we have $y \to \mathbf{s}$ and so $y \cdot \mathbf{s} \to 1$ and $r \to 1$.

The main idea of the proof is that when $x$ is close to $\mathbf{n}$, we divide $\overline{\B}^N$ into two parts: the ball at $0$ with some radius $\delta$ and the rest of $\overline{\B}^N$. We prove the following statement: For fixed $\delta >0$, there exist parameters $C(\delta)$ and $\alpha(\delta)$ such that $C(\delta) \to 1$ and $\alpha(\delta) \to 1$ as $\delta \to 0$ and
		\begin{align}
& \limsup_{x \to \mathbf{n}}\int_{\B^N \setminus \B^{N}(\delta) }\left(\frac{1-r^2}{1+r^2-2r  \cos \theta(y, \omega(z))}\right)^{\! N} |D\omega(z)|  \,   d \mathcal{H}_N(z) = 0  \label{intnotball}, \\
& \limsup_{x \to \mathbf{n}}\int_{\B^{N}(\delta)}\left(\frac{1-r^2}{1+r^2-2r \cos \theta(y, \omega(z))}\right)^{\! N} |D\omega(z)| \, d \mathcal{H}_N(z)\leq \left(\frac{C(\delta)}{ \alpha(\delta)} \right)^N \mathcal{H}_N({\mathbb S}^N),  \label{intball}
		\end{align}
where $y=-x/|x| \to \mathbf{s}$ and $r \to 1$ as $x \to \mathbf{n}$. Hence, the limsup of \autoref{intmob} as $x \to \mathbf{n}$ is at most $\left(C(\delta)/\alpha(\delta) \right)^N  \mathcal{H}_N({\mathbb S}^N) $, which tends to $ \mathcal{H}_N({\mathbb S}^N )$ as $\delta \to 1$.

To show \autoref{intnotball}, we use the fact that the angle $\theta(y, \omega(z))$ is bounded below away from 0 when $y$ is near $\mathbf{s}$ and $|z| \geq \delta$. Hence, the integrand in \autoref{intnotball} tends to 0 uniformly as $x \to \mathbf{n}$. 


To show \autoref{intball}, the first step is to show that the distance on the sphere $\theta(y, \omega(z))$ can be approximated by the Euclidean distance from $z(y)$ to $z$, where for $y$ close to $\mathbf s$, $z(y)$ is the unique point $z  \in \overline{\B}^{N}(\delta)$ such that $|y-\omega(z)|$ is minimal. That is, $\omega(z(y))$ is a closest point to $y$ on the surface, and such that $z(y)$ is unique for $y$ near the surface.

We claim that there exists $\alpha(\delta)>0$ such that	
	\begin{equation}
\frac{\theta(y, \omega(z))}{|z(y)-z|}\geq \alpha(\delta) \label{angle}
	\end{equation}
for all $y$ near $\mathbf{s}$ and $z \in \overline{\B}^N(\delta)$. We show that $\alpha(\delta)  \to 1$ as $\delta \to 0$. Since spherical distance is greater than Euclidean distance,
		\begin{align*}
\frac{\theta(y, \omega(z))}{|z(y)-z|} &\geq \frac{|y-\omega(z)|}{|z(y)-z|} \\
&\geq \beta(\delta)  \frac{|y-\omega(z)|}{|\omega(z(y))-\omega(z)|} ,
		\end{align*}
where the constant satisfies $\beta(\delta) \to 1$ as $\delta \to 0$, since $\beta(\delta)$ is a lower bound of $|\omega(z(y))-\omega(z)|/|z(y)-z|$ and $D\omega(0) =
\begin{pmatrix} 
I&0
\end{pmatrix}^T$.
Consider the triangle with vertices at $y$, $\omega(z(y))$ and $\omega(z)$. Write the sidelengths as $a=|y-\omega(z(y))|$, $b=|\omega(z(y))-\omega(z)|$, and $c =|y-\omega(z)|$, so that the last quantity displayed is $\beta(\delta) c/b$. By the law of sines,
$$
c^2 \geq b^2 \sin^2 \eta, 
$$
where the angle $\eta$ at point $\omega(z(y))$ is close to $\pi/2$ because the side from $\omega(z(y))$ to $y$ is normal to the surface and the vector from $\omega(z(y))$ to $\omega(z)$ is nearly tangential to the surface. In this case, our lower bound $\alpha(\delta)$ is defined as $\beta(\delta)  c/b  \geq \beta(\delta)  | \sin \eta| := \alpha(\delta)$, where $\alpha(\delta) \to 1$ because $\beta(\delta) \to 1$ and $\eta(\delta) \to \pi/2$ as $\delta \to 0$. This finishes the proof of \autoref{angle}. 


Based on the lower bound \autoref{angle}, the second step is to show that the integral in \autoref{intball} has the following upper bound, 
		\begin{align*}
&\int_{\B^{N}(\delta)}\left(\frac{1-r^2}{1+r^2-2r \cos \theta(y, \omega(z))}\right)^{\! N}|D\omega(z)|\,  d \mathcal{H}_N\\
&\leq \int_{\B^N(\delta)}\left(\frac{1-r^2}{1+r^2-2r\cos (\alpha(\epsilon)|z-z(y)|)}\right)^{\! N} |D\omega(z)|\,   d \mathcal{H}_N\\
&\leq C(\delta)^N \int_{\B^N(\delta)}\left(\frac{1-r^2}{1+r^2-2r \cos (\alpha(\epsilon)|z-z(y)|)}\right)^{\! N}  \, d \mathcal{H}_N,
		\end{align*}
where $C(\delta)$ denotes the maximum of $|D\omega(z)|$ on $\B^N(\delta)$. Note $C(\delta) \to 1$ as $\delta \to 0$. Next, we change the variable with $\alpha(\delta)(z-z(y)) = (1-r)\zeta$ to write the upper bound as follows,
		\begin{align*}
&\int_{\B^{N}(\delta)}\left(\frac{1-r^2}{1+r^2-2r \cos \theta(y, \omega(z))}\right)^{\! N} |D\omega(z)| \,  d \mathcal{H}_N  \\
&\leq \left(\frac{C(\delta)}{\alpha(\delta)}\right)^N\int_{\B^N \left(2 \delta \alpha(\delta)/ (1-r) \right)} \left(\frac{1-r^2}{1+r^2 -2r\cos (1-r)|\zeta|}\right)^{\! N} \left(1-r \right)^N \,   d \mathcal{H}_N .					
		\end{align*}
The domain approaches to all of $\R^N$ as $r\to 1$. By considering the Maclaurin series, we use dominated convergence and find the limit of the upper bound. Note that $\cos(1-r)|\zeta| \leq 1- (1-r)^2|\zeta|^2/2 + (1-r)^4  O(|\zeta|^4)$, since $(1-r)|\zeta| \leq 2\delta \alpha(\delta)$ is small. Then the last integral has the following upper bound:
		\begin{align*}
&\int_{\B^N \left(2 \delta \alpha(\delta)/ (1-r) \right)} \left(\frac{(1-r)^2(1+r)}{1+r^2-2r( 1- (1-r)^2|\zeta|^2/2 + (1-r)^4  O(|\zeta|^4))}\right)^{\! N} \,   d \mathcal{H}_N (\zeta) \\
& =\int_{\B^N \left(2 \delta \alpha(\delta)/ (1-r) \right)}\left(\frac{(1-r)^2(1+r)}{(1-r)^2 + r(1-r)^2|\zeta|^2- (1-r)^4 O(|\zeta|^4)}\right)^{\! N}  \,   d \mathcal{H}_N(\zeta)\\
& = \int_{\B^N \left(2 \delta \alpha(\delta)/ (1-r) \right)}\left(\frac{1+r}{1 + |\zeta|^2(r- (1-r)^2 O(|\zeta|^2))}\right)^N  \,  d \mathcal{H}_N(\zeta)\\
& \to \int_{\R^N}\left( \frac{2}{1+|\zeta|^2}\right)^{\! N}  \,  d \mathcal{H}_N(\zeta)
		\end{align*}
as $r \to 1$, by dominated convergence. The last integral is equal to $\mathcal{H}_N({\mathbb S}^N)$.

Lastly, assume that south pole $\mathbf{s} \notin \omega(\B^N)$.  In case when $\mathbf{s} \notin \omega(\overline{\B}^N)$, there exists $\delta$ small enough so that the volume vanishes in the limit.  In case when $\mathbf{s} \in \omega(\partial \overline{\B}^N)$, consider a bigger ball containing $\overline{\B}^N$ to enlarge the surface $\omega(\overline{\B}^N)$. The south pole lies in the interior of this bigger ball and so the volume satisfies the conclusion \autoref{txdeg} by our work above. 

\end{proof}

\appendix\section{Calculations}
Here, we check that ratio of the right-hand sides of \autoref{mainequ} and \autoref{mainequ2} approaches 1 as the dimension $n$ tends to $\infty$.
\begin{lemma} 
\label{ratio}
For $n \geq 2$, let us denote $A_n= ((2n+2)^{n/2}+2n^{n/2})^{2/n}$ and $B_n= 2^{2/n}(2n+2)$. Then $A_n/B_n$ is strictly smaller than 1 and approaches 1 as $n \to \infty$.
\end{lemma}

\begin{proof}
Note that $A_n/B_n < 1$ for all $n \geq 2$ because 
\begin{align*}
 A_n  &= \left((2n+2)^{n/2}+2n^{n/2}\right)^{2/n}   
\\& < \left(2(2n+2)^{n/2} \right)^{2/n} \qquad \text{since $2 \leq 2^{n/2}$ and $2n< 2n+2$}
\\& = 2^{2/n}(2n+2) = B_n.
\end{align*}
Further, $A_n/B_n  \geq 2^{-2/n}  \to 1$ as $n \to \infty$, since
\begin{align*}
 A_n &=\left((2n+2)^{n/2}+2n^{n/2}\right)^{2/n}   
\\& > \left((2n+2)^{n/2}+(n+1)^{n/2}\right)^{2/n} \qquad \text{by the inequality $2^{2/n}n > n+1  $ for $n \geq 2$}
\\& = \left(2^{n/2}+1\right)^{2/n} (n+1) > 2(n+1) = 2^{-2/n}  B_n.
\end{align*}
\end{proof}

\section*{\textbf{Acknowledgements}}
I am grateful for support from the University of Illinois through the Campus Research Board award RB22004 (to Richard Laugesen), the Lois M. Lackner fellowship from the Department of Mathematics, and the National Science Foundation award \#2246537 (to Richard Laugesen).

\bibliographystyle{plain}


\end{document}